\documentclass[a4paper,12pt]{amsart}
\usepackage{amsfonts}
\usepackage{amsmath}
\usepackage{amssymb}
\usepackage[a4paper]{geometry}
\usepackage{mathrsfs}
\usepackage{hyperref}
\renewcommand\eqref[1]{(\ref{#1})} 
\numberwithin{equation}{section}
\theoremstyle{plain}
\newtheorem{thm}{Theorem}[section]

\newtheorem{lem}[thm]{Lemma}

\theoremstyle{definition}

\newtheorem{rem}[thm]{Remark}

\begin{document}
\title[Rayleigh-Faber-Krahn inequality on $\mathbb{S}^{n}$ and $\mathbb{H}^{n}$]
{On first and second eigenvalues of Riesz transforms in spherical and hyperbolic geometries}

\author[Michael Ruzhansky]{Michael Ruzhansky}
\address{
  Michael Ruzhansky:
  \endgraf
  Department of Mathematics
  \endgraf
  Imperial College London
  \endgraf
  180 Queen's Gate, London SW7 2AZ
  \endgraf
  United Kingdom
  \endgraf
  {\it E-mail address} {\rm m.ruzhansky@imperial.ac.uk}
  }
\author[Durvudkhan Suragan]{Durvudkhan Suragan}
\address{
  Durvudkhan Suragan:
 \endgraf
  Institute of Mathematics and Mathematical Modelling
  \endgraf
  125 Pushkin str., 050010 Almaty, Kazakhstan
  \endgraf
  and
  \endgraf
  Department of Mathematics
  \endgraf
  Imperial College London
  \endgraf
  180 Queen's Gate, London SW7 2AZ
  \endgraf
  United Kingdom
  \endgraf
  {\it E-mail address} {\rm d.suragan@imperial.ac.uk}}

\thanks{The authors were supported in parts by the EPSRC grant EP/K039407/1 and by the Leverhulme Grant RPG-2014-02,
as well as by the MESRK grant 5127/GF4.}

 \keywords{Convolution operators, $n$-sphere, real hyperbolic space,
 Rayleigh-Faber-Krahn inequality, Hong-Krahn-Szeg{\"o} inequality}
 \subjclass{35P99, 47G40, 35S15}

\begin{abstract}
In this note we prove an analogue of the Rayleigh-Faber-Krahn inequality, that is, that the geodesic ball is a maximiser of the first eigenvalue of some convolution type integral operators, on the sphere $\mathbb{S}^{n}$ and on the real hyperbolic space $\mathbb{H}^{n}$.
It completes the study of such question for complete, connected, simply connected Riemannian manifolds of constant sectional curvature.
We also discuss an extremum problem for the second eigenvalue on $\mathbb{H}^{n}$ and prove the Hong-Krahn-Szeg\"{o} type inequality.
The main examples of the considered convolution type operators are the Riesz transforms with
respect to the geodesic distance of the space.
\end{abstract}

\dedicatory{Dedicated to Professor T. Sh. Kalmenov on the occasion of his 70$^{th}$ anniversary}

     \maketitle
\section{Introduction}
\label{intro}

Let $M$ be a complete, connected, simply connected Riemannian manifolds of constant sectional curvature. As it is known the three possibilities for $M$ are $\mathbb{S}^{n}, \mathbb{R}^{n}$
and $\mathbb{H}^{n}$ (the sphere, the Euclidean space, and the real hyperbolic space)
for positive, zero, and negative curvature, respectively.
In this paper we are first interested in the Riesz transform on such spaces, namely, in the
operator
\begin{equation}
\mathcal{R}_{\alpha}f(x):= \int_{M}\frac{1}{d(x,y)^{\alpha}}f(y)dy,
\label{Riesz}
\end{equation}
where $dy$ is the Riemannian measure on $M$ and $d(x,y)$ is the geodesic distance.
Restricting $f$ to be compactly supported in an open bounded set $\Omega\subset M$,
we can consider the family of operators
\begin{equation}
\mathcal{R}_{{\alpha},{\Omega}}f(x):= \int_{\Omega}\frac{1}{d(x,y)^{\alpha}}f(y)dy,\quad
0<\alpha<n,
\label{Riesz-Omega}
\end{equation}
depending on $\Omega$. In this paper we are interested in the behaviour of the first and
second eigenvalues of operators $\mathcal{R}_{{\alpha},{\Omega}}$.
Historically, in his famous
book ``Theory of Sound'' (first published in 1877), by using explicit computation
and physical arguments, Lord Rayleigh stated that a disk minimises (among all
domains of the same area) the first eigenvalue of the Dirichlet Laplacian.
The proof of this conjecture was obtained about 50 years later, simultaneously (and independently) by G.
Faber and E. Krahn. Nowadays, the Rayleigh-Faber-Krahn inequality has been
extended to many other operators, see e.g. \cite{He} for
further references. A recent general review of isoperimetric inequalities for the Dirichlet, Neumann and other Laplacians on Euclidean spaces was made by Benguria, Linde and Loewe in \cite{Ben}. We also refer to Henrot \cite{He} and Brasco and Franzina \cite{BF} for more historic remarks on isoperimetric inequalities, namely the Rayleigh-Faber-Krahn inequality and the Hong-Krahn-Szeg\"{o} inequality. 

In this note we establish a similar result also on $\mathbb S^{n}$ and $\mathbb H^{n}$ thus completing the picture now for all complete, connected, simply connected Riemannian manifolds of constant sectional curvature.
In the case of the real hyperbolic space $\mathbb H^{n}$ we also establish the
analogue of the Hong-Krahn-Szeg\"{o} inequality on $\mathbb R^{n}$, namely, the description of
$\Omega$ for which the second eigenvalue is maximised.
In fact, our results apply to a more general class of convolution type operators than the Riesz transforms \eqref{Riesz-Omega} that we will describe further.

For the Riesz potential operators on $\mathbb R^{n}$ results analogous to those of the present note have been obtained by G. Rozenblum and the authors in \cite{RRS} (see also \cite{Ruzhansky-Suragan-log} for the logarithmic potential operator). Thus, here
we restrict our attention to $\mathbb S^{n}$ and $\mathbb H^{n}$.
So, let $M$ denote $\mathbb S^{n}$ or $ \mathbb H^{n}$.

Let $\Omega\subset \mathbb S^{n}$ or $\Omega\subset \mathbb H^{n}$ be an open bounded set.
We consider the integral operator $\mathcal{K}_{\Omega}:L^{2}(\Omega)\rightarrow L^{2}(\Omega)$ defined by
\begin{equation}
\mathcal{K}_{\Omega}f(x):= \int_{\Omega}K(d(x,y))f(y)dy,\quad f\in L^{2}(\Omega),
\label{3}
\end{equation}
which we assume to be compact.
Here $d(x,y)$ is the distance between the points $x$ and $y$ in the symmetric space $\mathbb S^{n}$ or $\mathbb H^{n}$.
Throughout this note we assume that the kernel $K(\cdot)$ is
(say, a member of $L^{1}(\mathbb S^{n})$ or $L^{1}_{loc}(\mathbb H^{n})$) real, positive and non-increasing, i.e.
that the function $K:[0,\infty)\to \mathbb R^{+}$ satisfies
\begin{equation}\label{n1}
K(\rho_{1})\geq K(\rho_{2})\quad  \textrm{ if } \quad \rho_{1}\leq \rho_{2}.
\end{equation}
Since $K$ is a real and symmetric function, $\mathcal{K}_{\Omega}$ is a self-adjoint operator. Therefore, all of its eigenvalues are real.
The eigenvalues  of $\mathcal{K}_{\Omega}$ may be enumerated in the descending order of their moduli,
\begin{equation} \label{EQ:ordering}
|\lambda_{1}|\geq|\lambda_{2}|\geq...,
\end{equation}
where $\lambda_{j}=\lambda_{j}(\Omega)$ is repeated in this series according to its multiplicity.

We denote the
corresponding eigenfunctions by $u_{1}, u_{2},...,$ so that for each eigenvalue $\lambda_{j}$ there is a unique corresponding (normalised) eigenfunction $u_{j}$:
$$\mathcal{K}_{\Omega}u_{j}=\lambda_{j}(\Omega)u_{j},\,\,\,\, j=1,2,\ldots.$$
In this paper we are interested in isoperimetric inequalities of the convolution type operator $\mathcal{K}_{\Omega}$ for the first and the second eigenvalues. The main reason why the results are useful, beyond the intrinsic interest of geometric extremum problems, is that they produce \emph{a priori} bounds for spectral invariants of operators on arbitrary domains (see, for example, Remark \ref{REM:op}). For a further discussion also of higher eigenvalues in terms of
the integral kernel we refer to \cite{DR} and references therein.

In this note we prove the Rayleigh-Faber-Krahn inequality for the integral operator
$\mathcal{K}_{\Omega}$, i.e.
it is proved (in Theorem \ref{THM:1})
that the geodesic ball is a maximiser of the first eigenvalue of the integral
operator $\mathcal{K}_{\Omega}$ among all domains of a given measure in $M$.
The proof is based on the Riesz-Sobolev inequality on $M$ (established in \cite{beck}) allowing the use of the
symmetrization techniques.

As we mentioned above, for the Riesz potential operators on $\mathbb R^{n}$ some results analogous to those of the present note have been obtained by the authors and G. Rozenblum in \cite{RRS}. See also \cite{RS-UMN} for an announcement. Here we extend those results to the $n$-sphere and to the real hyperbolic space. Although there is some overlap between these settings we give complete proofs here to demonstrate our techniques on $\mathbb S^{n}$ and $\mathbb H^{n}$.
Summarising our results for operators $\mathcal{K}_{\Omega}$ in both cases of $\mathbb S^{n}$ or $\mathbb H^{n}$, we prove the following facts:
\begin{itemize}
\item Rayleigh-Faber-Krahn type inequalities: the first eigenvalue of $\mathcal{K}_{\Omega}$ is maximised
on the geodesic ball among all domains of a given measure in $\mathbb S^{n}$ or $\mathbb H^{n}$;
\item Hong-Krahn-Szeg\"{o} type inequality: the maximum of the second eigenvalue of (positive) $\mathcal{K}_{\Omega}$ among bounded open sets with a given measure in $\mathbb H^{n}$ is achieved by the union of two identical geodesic balls with mutual distance going to infinity.
\end{itemize}

In Section \ref{SEC:result} we present main results of this note.
Their proof will be given in Sections \ref{sec:3} and \ref{SEC:4}.

\section{Main results}
\label{SEC:result}

As outlined in the introduction, we assume that
$\Omega\subset M$
is an open bounded set, and we consider compact integral operators
on $L^{2}(\Omega)$ of the form
\begin{equation}
\mathcal{K}_{\Omega}f(x)=\int_{\Omega}K(d(x,y))f(y)dy,\quad f\in L^{2}(\Omega),
\label{n16}
\end{equation}
where the kernel $K$ is real, positive and non-increasing, that is, $K$
satisfies \eqref{n1}.
By $|\Omega|$ we will denote the Riemannian measure of $\Omega$.
We prove the following analogue of the Rayleigh-Faber-Krahn inequality for the integral operator $\mathcal{K}_\Omega$.
\begin{thm} \label{THM:1}
Let $M$ denote $\mathbb S^{n}$ or $\mathbb H^{n}$. The geodesic ball $\Omega^{*}\subset M$ is a
maximiser of the first eigenvalue of the operator $\mathcal{K}_{\Omega}$ among all domains of a given measure in $M$, i.e.
\begin{equation}
\label{EQ:mu1}
0<\lambda_{1}(\Omega{})\leq \lambda_{1}(\Omega^*)
\end{equation}
for an arbitrary domain $\Omega\subset M$ with $|\Omega|=|\Omega^{*}|.$
\end{thm}

\begin{rem}\label{REM:op}
In other words Theorem \ref{THM:1} says that the operator norm
$\|\mathcal{K}_\Omega\|_{{\mathscr L}(L^{2}(\Omega))}$ is maximised in a geodesic ball among all domains of a given measure. We also note that since $\lambda_{1}$ is the eigenvalue with the largest modulus according to the ordering
\eqref{EQ:ordering}, we will show in Lemma \ref{lem:1} that $\lambda_{1}$ is actually
positive. Therefore, \eqref{EQ:mu1} is the inequality between positive numbers.
\end{rem}

We are also interested in maximising the second eigenvalue of positive operators $\mathcal{K}_{\Omega}$ on $\mathbb H^{n}$ among open sets of given measure.

\begin{thm}\label{THM:second}
If the kernel $K$ of the positive operators $\mathcal{K}_{\Omega}$ on $\mathbb H^{n}$ satisfies
\begin{equation}
\label{EQ:kinfty}
K(\rho)\rightarrow 0\;\textrm{ as }\; \rho\rightarrow \infty,
\end{equation}
then
the maximum of $\lambda_{2}$ among bounded open sets in $\mathbb H^{n}$ of a given measure is achieved by the  union of two identical geodesic balls with mutual distance going to infinity. Moreover, this maximum is equal to the first eigenvalue of one of the two geodesic balls.
\end{thm}

A similar type of results on $\mathbb R^{n}$ is called the Hong-Krahn-Szeg\"{o} inequality. See, for example, \cite{BF} for further references. We note that in Theorem \ref{THM:second} we consider only domains $\Omega\subset\mathbb H^{n}$ for which $\mathcal{K}_{\Omega}$ are positive operators. However, this can be relaxed:

\begin{rem}
The statement of Theorem \ref{THM:second} and its proof remain valid if we only assume that the second eigenvalues $\lambda_{2}(\Omega)$ of considered operators $\mathcal{K}_{\Omega}$ are positive.
\end{rem}

\begin{rem}
We note that the proofs in the sequel work equally well also in $\mathbb R^{n}$ and
the statements of Theorem \ref{THM:1} and Theorem \ref{THM:second} are valid with
$\mathbb H^{n}$ replaced by $\mathbb R^{n}$. See also \cite{RS-UMN} for the announcement.

In the case of $\mathbb R^{n}$ and Riesz transforms \eqref{Riesz-Omega}, Theorem
\ref{THM:second} holds without the
positivity assumption since the Riesz transforms $\mathcal{R}_{{\alpha},{\Omega}}$ on
$\mathbb R^{n}$ are positive, see \cite{RRS}.

We do not have a version of Theorem \ref{THM:second} on the spheres
$\mathbb S^{n}$ because the assumption \eqref{EQ:kinfty} does not make sense
due to compactness of the sphere.
\end{rem}

\section{Proof of Theorem \ref{THM:1}}
\label{sec:3}

Since the integral kernel of $\mathcal{K}$ is positive, the following statement, sometimes called Jentsch's theorem, applies. However, for completeness of this note we restate and give its proof on the symmetric space $M$ (that is, $\mathbb{S}^{n}$ or $\mathbb{H}^{n}$).

\begin{lem}\label{lem:1}
The first eigenvalue $\lambda_{1}$ (with the largest modulus) of the convolution type compact operator $\mathcal{K}$ is positive and simple; the corresponding eigenfunction $u_{1}$ can be chosen positive.
\end{lem}

\begin{proof}[Proof of Lemma \ref{lem:1}]
It is well known that for the self-adjoint compact operators the eigenvalues are real, and the corresponding eigenfunctions form a complete orthogonal basis on $L^{2}$.
The eigenfunctions of the convolution type compact operator $\mathcal{K}$
may be chosen to be real as its kernel is real. First let us prove that $u_{1}$ cannot change sign in the domain
$\Omega\subset M$, that is,
$$u_{1}(x)u_{1}(y)=|u_{1}(x)u_{1}(y)|,\,\,x,y\in\Omega\subset M.$$

In fact, in the opposite case, in view of the continuity of the function
$u_{1}(x)$, there would be neighborhoods $U(x_{0},r)\subset \Omega$ and
$U(y_{0},r)\subset \Omega$ such that

$$|u_{1}(x)u_{1}(y)|> u_{1}(x)u_{1}(y),\,\,x\in U(x_{0},r)\subset\Omega,\,\,y\in U(y_{0},r)\subset\Omega,$$
and so in view of
\begin{equation}\label{n0}
\int_{\Omega}K(d(x,z))K(d(z,y))dz>0
\end{equation}
we obtain
\begin{multline} \label{n2}
\frac{(\mathcal{K}^{2}|u_{1}|,\,|u_{1}|)}{\|u_{1}\|^{2}}=\frac{1}{\| u_{1}\|^{2}}\int_{\Omega}
\int_{\Omega}\int_{\Omega}K(d(x,z))K(d(z,y))dz |u_{1}(x)||u_{1}(y)|dxdy \\
>\frac{1}{\| u_{1}\|^{2}}\int_{\Omega}\int_{\Omega}\int_{\Omega}
K(d(x,z))K(d(z,y))dz u_{1}(x)u_{1}(y)dxdy=\lambda_{1}^{2}.
\end{multline}
We note that
$\lambda_{1}^{2}$ is the largest eigenvalue of $\mathcal{K}^{2}$ and $u_{1}$
is the eigenfunction corresponding to $\lambda_{1}^{2}$, i.e.
$$\lambda_{1}^{2}u_{1}=\mathcal{K}^{2}u_{1}.$$
Therefore, by the variational principle we have
\begin{equation}\label{n3}
\lambda_{1}^{2}=\sup_{f\in L^{2}(\Omega), f\not\equiv 0}\frac{(\mathcal{K}^{2}f,f)}{\| f\|^{2}}.
\end{equation}
This means that the strict inequality \eqref{n2} contradicts the variational
principle \eqref{n3}.

Now we shall prove that the eigenfunction $u_{1}(x)$ cannot become zero
in $\Omega$ and therefore can be chosen positive in $\Omega$.
In fact, in the opposite case there would be a point
$x_{0}\in \Omega$ such that
$$
0=\lambda_{1}^{2}u_{1}(x_{0})=\int_{\Omega}
\int_{\Omega}K(d(x_{0},z))K(d(z,y))dz\, u_{1}(y)dy,
$$
from which, in view of condition \eqref{n0},
the contradiction follows: $u_{1}(y)=0$ for almost all $y\in \Omega.$

Since $u_{1}$ is positive it follows that
$\lambda_{1}$ is a simple. In fact, if there were an eigenfunction
$\widetilde{u}_{1}$ linearly independent of $u_{1}$ and corresponding
to $\lambda_{1}$, then for all real $c$ every linear combination
$u_{1}+c\widetilde{u}_{1}$ would also be an eigenfunction corresponding
to $\lambda_{1}$ and therefore, by what has been proved, it could not become zero
in $\Omega$. As $c$ is arbitrary, this is impossible.
Finally, it remains to show $\lambda_{1}$ is positive.
It is trivial since $u_{1}$ and the kernel are positive.
\end{proof}

\begin{proof}[Proof of Theorem \ref{THM:1}]
Let $\Omega$ be a bounded measurable set in $M$ (where, as above, $M$ is $\mathbb S^{n}$ or $\mathbb H^{n}$).
Its symmetric rearrangement $\Omega^{\ast}$ is an open geodesic ball centred at $0$
with the measure equal to the measure of $\Omega,$ i.e. $|\Omega^{\ast}|=|\Omega|$.
Let $u$ be a nonnegative measurable function in $\Omega$
such that all its positive level sets have finite measure.
Here we give an abstract definition of the symmetric-decreasing rearrangement of $u$ (we refer \cite{BT} and \cite{beck1} for more detailed discussions on this subject):
Let $u$ be a nonnegative measurable function in $\Omega\subset M$. The function
\begin{equation}
u^{\ast}(x):=\intop_{0}^{\infty}\chi_{\left\{u(x)>t\right\}^{\ast}\label{2}}dt
\end{equation}
is called the (radially) symmetric-decreasing rearrangement of a nonnegative measurable function $u$.

By Lemma \ref{lem:1} the first eigenvalue $\lambda_{1}$ of the operator $\mathcal{K}$ is simple; the corresponding
eigenfunction $u_{1}$ can be chosen positive in $\Omega\subset M$.
Recall the Riesz-Sobolev inequality (see e.g., Symmetrization Lemma in \cite{beck}):
\begin{equation}\label{n10}
\int_{\Omega}\int_{\Omega}
u_{1}(y)K(d(y,x))u_{1}(x)dydx
\leq
\int_{\Omega^{\ast}}\int_{\Omega^{\ast}}
u_{1}^{\ast}(y)K(d(y,x))u_{1}^{\ast}(x)dydx.
\end{equation}
In addition, for each nonnegative function $u\in L^2(\Omega)$ we have
\begin{equation}\label{n11}
\| u\|_{L^2(\Omega)}=\| u^{\ast}\|_{L^2(\Omega^{\ast})}.
\end{equation}
Therefore, from \eqref{n10}, \eqref{n11} and the variational principle for $\lambda_{1}(\Omega^{\ast})$, we get
$$
\lambda_{1}(\Omega) =\frac{\int_{\Omega}\int_{\Omega}
u_{1}(y)K(d(y,x))u_{1}(x)dydx}
{\int_{\Omega}|u_{1}(x)|^{2}dx}\leq\frac{\int_{\Omega^{\ast}}\int_{\Omega^{\ast}}
u^{\ast}_{1}(y)K(d(y,x))
u^{\ast}_{1}(x)dydx}{\int_{\Omega^{\ast}}|u^{\ast}_{1}(x)|^{2}dx}
$$
$$
\leq\sup_{v\in
L^2(\Omega^{\ast}),v\neq 0}\frac{\int_{\Omega^{\ast}}\int_{\Omega^{\ast}}
v(y)K(d(y,x))
v(x)dydx}{\int_{\Omega^{\ast}}|v(x)|^{2}dx}={\lambda_{1}(\Omega^{\ast})},
$$
completing the proof.
\end{proof}

\section{Proof of Theorem \ref{THM:second}}
\label{SEC:4}

Proof of Theorem \ref{THM:second} is similar to the case of $\mathbb{R}^{n}$.
However, the proof does not work for $\mathbb{S}^{n}$ since we use the decay property of the kernel at infinity.

\begin{proof}[Proof Theorem \ref{THM:second}] Let us introduce the following sets:
$$\Omega^{+}:=\{x: u_{2}(x)>0\},\,\,\Omega^{-}:=\{x: u_{2}(x)<0\}.$$
Therefore,
$$u_{2}(x)\gtrless0,\,\, \forall x\in\Omega^{\pm}\subset \Omega\subset \mathbb{H}^{n}, \,\, \Omega^{\pm}\neq\{\emptyset\},$$
and it follows from Lemma \ref{lem:1} that
the domains $\Omega^{-}$ and $\Omega^{+}$ both have a positive measure.
Taking
\begin{equation}
u_{2}^{\pm}(x):=\left\{
\begin{array}{ll}
    u_{2}(x)\,\,{\rm in}\,\, \Omega^{\pm},\\
    0 \,\, {\rm otherwise}, \\
\end{array}
\right.
\end{equation}
we obtain
$$\lambda_{2}(\Omega)u_{2}(x)=\int_{\Omega^{+}}K(d(x,y))u_{2}^{+}(y)dy+
\int_{\Omega^{-}}K(d(x,y))u_{2}^{-}(y)dy,\,\, x\in\Omega.$$
Multiplying by $u_{2}^{+}(x)$ and integrating over $\Omega^{+}$ we get

\begin{multline*}
\lambda_{2}(\Omega)\int_{\Omega^{+}}|u_{2}^{+}(x)|^{2}dx=
\int_{\Omega^{+}}u_{2}^{+}(x)\int_{\Omega^{+}}K(d(x,y))u_{2}^{+}(y)dydx \\
+\int_{\Omega^{+}}u_{2}^{+}(x)
\int_{\Omega^{-}}K(d(x,y))u_{2}^{-}(y)dydx,\,\, x\in\Omega.
\end{multline*}
The second term on the right hand side is non-positive since  the integrand is non-positive.
Therefore,

$$\lambda_{2}(\Omega)\int_{\Omega^{+}}|u_{2}^{+}(x)|^{2}dx\leq
\int_{\Omega^{+}}u_{2}^{+}(x)\int_{\Omega^{+}}K(d(x,y))u_{2}^{+}(y)dydx,$$
that is,
$$\frac{\int_{\Omega^{+}}u_{2}^{+}(x)\int_{\Omega^{+}}K(d(x,y))u_{2}^{+}(y)dydx}{\int_{\Omega^{+}}|u_{2}^{+}(x)|^{2}dx}\geq\lambda_{2}(\Omega).$$
By the variational principle,
$$\lambda_{1}(\Omega^{+})=
\sup_{v\in L^{2}(\Omega^{+}), v\not\equiv 0}\frac{\int_{\Omega^{+}}v(x)\int_{\Omega^{+}}K(d(x,y))v(y)dydx}{\int_{\Omega^{+}}|v(x)|^{2}dx}$$
$$\geq\frac{\int_{\Omega^{+}}u_{2}^{+}(x)\int_{\Omega^{+}}K(d(x,y))u_{2}^{+}(y)dydx}{\int_{\Omega^{+}}|u_{2}^{+}(x)|^{2}dx}\geq\lambda_{2}(\Omega).$$
Similarly, we get
$$\lambda_{1}(\Omega^{-})\geq\lambda_{2}(\Omega).$$
So we have
\begin{equation}\label{13}
\min\{\lambda_{1}(\Omega^{+}),
\lambda_{1}(\Omega^{-})\}\geq \lambda_{2}(\Omega).
\end{equation}
We now introduce $B^{+}$ and $B^{-}$, the geodesic balls of the same volume as $\Omega^{+}$ and $\Omega^{-}$, respectively.
Due to Theorem \ref{THM:1}, we have
\begin{equation}\label{14}
\lambda_{1}(B^{+})\geq\lambda_{1}(\Omega^{+}),
\,\,\lambda_{1}(B^{-})\geq\lambda_{1}(\Omega^{-}).
\end{equation}
Combining \eqref{13} and \eqref{14}, we obtain
\begin{equation}\label{15}
\min \{\lambda_{1}(B^{+}),\,\lambda_{1}(B^{-})\}\geq \lambda_{2}(\Omega).
\end{equation}
Now let us consider the set  $B^{+}\cup B^{-},$ with the geodesic balls $B^{+}$ and $B^{-}$ placed  at distance $l$, i.e.
$$l= {\rm dist}(B^{+},B^{-}),$$
Denote by $u_1^\circledast$ the first normalised eigenfunction of $\mathcal{K}_{B^{+}\cup B^{-}}$ and take $u_{+}$ and $u_{-}$ being the first normalised eigenfunctions of each single geodesic ball, i.e., of operators $\mathcal{K}_{B^{\pm}}.$  We introduce the function $v^\circledast\in L^2(B^{+}\cup B^{-})$, which equals $u_+$ in $B^{+}$ and $\gamma u_-$ in $B^{-}$. Since the functions $u_+,u_-, u^\circledast$ are positive, it is possible to find a real number $\gamma$ so that $v^\circledast$ is orthogonal to $u_1^\circledast$.
 Observe that
\begin{equation}
\int_{B^{+}\cup B^{-}}\int_{B^{+}\cup B^{-}} v^\circledast(x)v^\circledast(y)
K(d(x,y))dxdy
=\sum_{i=1}^{4}\mathcal{I}_{i},
\end{equation}
where
$$\mathcal{I}_{1}:=\int_{B^{+}}\int_{B^{+}}u_{+}(x)u_{+}(y)
K(d(x,y))dxdy,\, \mathcal{I}_{2}:=\gamma\int_{B^{+}}\int_{B^{-}}u_{+}(x)u_{-}(y)
K(d(x,y))dxdy,$$
$$\mathcal{I}_{3}:=\gamma\int_{B^{-}}\int_{B^{+}}u_{-}(x)u_{+}(y)
K(d(x,y))dxdy,$$ $$\mathcal{I}_{4}:=\gamma^{2}\int_{B^{-}}\int_{B^{-}}u_{-}(x)u_{-}(y)
K(d(x,y))dxdy.$$
By the variational principle,
$$\lambda_{2}(B^{+}\cup B^{-})= \sup_{v\in L^{2}(B^{+}\bigcup B^{-}),\,v\perp u_{1},\,\parallel v\parallel=1}{\int_{B^{+}\cup B^{-}}\int_{B^{+}\cup B^{-}}v(x)v(y)K(d(x,y))dxdy}.$$
Since by construction $v^\circledast$ is orthogonal to $u_{1}$,
we get
$$\lambda_{2}(B^{+}\cup B^{-})\geq {\int_{B^{+}\cup B^{-}}\int_{B^{+}\cup B^{-}} v^\circledast(x)v^\circledast(y) K(d(x,y))dxdy}
= {\sum_{i=1}^{4}\mathcal{I}_{i}}.$$
On the other hand, since $u_{+}$ and $u_{-}$ are the first normalised eigenfunctions (by Lemma \ref{lem:1} both are positive everywhere) of each single geodesic ball $B^{+}$ and $B^{-}$, we have

$$\lambda_1(B^{\pm})=\int_{B^{\pm}}\int_{B^{\pm}}
u_{\pm}(x)u_{\pm}(y)K(d(x,y))dxdy.$$
Summarising the above facts, we obtain
\begin{equation}
\lambda_{2}(B^{+}\cup B^{-})\geq{\sum_{i=1}^{4}\mathcal{I}_{i}}\geq\frac{\sum_{i=1}^{4}\mathcal{I}_{i}}{1+\gamma^{2}}=
\frac{\mathcal{I}_{1}+\mathcal{I}_{4}+\mathcal{I}_{2}+\mathcal{I}_{3}}{\lambda_{1}(B^{+})^{-1}
\mathcal{I}_{1}+\lambda_{1}(B^{-})^{-1}\mathcal{I}_{4}}.
\end{equation}
Since the kernel $K(d(x,y))$ tends to zero as $x\in B^{\pm}, \ y\in B^{\mp}$
and $l\to\infty$,  we observe that
$$\lim_{l\rightarrow\infty}\mathcal{I}_{2}=\lim_{l\rightarrow\infty}\mathcal{I}_{3}=0,$$
thus
\begin{equation}\label{16}
\lim_{l\rightarrow\infty}\lambda_{2}(B^{+}\bigcup B^{-})\geq \max \{\lambda_{1}(B^{+}),\,\lambda_{1}(B^{-})\},
\end{equation}
where $l= {\rm dist}(B^{+},B^{-}).$
The inequalities \eqref{15} and \eqref{16} imply that the optimal set for $\lambda_{2}$ does not exist.
On the other hand, taking $\Omega\equiv B^{+}\bigcup B^{-}$ with $l= {\rm dist}(B^{+},B^{-})\rightarrow\infty$, and $B^{+}$ and $B^{-}$ being identical, from the inequalities \eqref{15} and \eqref{16}
we obtain
\begin{multline}\label{l21}
\lim_{l\rightarrow\infty}\lambda_{2}(B^{+}\bigcup B^{-})\geq \min \{\lambda_{1}(B^{+}),\,\lambda_{1}(B^{-})\}=\lambda_{1}(B^{+})
\\
=\lambda_{1}(B^{-})\geq\lim_{l\rightarrow\infty}\lambda_{2}(B^{+}\cup B^{-}),
\end{multline}
and this implies that the maximising sequence for $\lambda_{2}$ is given by a disjoint union of two identical geodesic balls with mutual distance going to $\infty$.
\end{proof}

\end{document}